\documentclass[12pt]{amsart} 
\usepackage{amssymb,latexsym}
\usepackage{amsmath}
\usepackage{graphicx}
\usepackage{color}
\usepackage{pinlabel}
\newtheorem{thm}{Theorem}[section]  
\newtheorem{cor}[thm]{Corollary}
\newtheorem{defin}[thm]{Definition} 
\newtheorem{lemma}[thm]{Lemma} 
\newtheorem{example}[thm]{Example} 
\newtheorem{prop}[thm]{Proposition}

\begin{document}  

\newcommand{\aaa}{\mbox{$\alpha$}}
\newcommand{\map}{\mbox{$\rightarrow$}}
\newcommand{\kkk}{\mbox{$\kappa$}}
\newcommand{\bbb}{\mbox{$\beta$}}
\newcommand{\sss}{\mbox{$\sigma$}}  
\newcommand{\ddd}{\mbox{$\delta$}} 
\newcommand{\rrr}{\mbox{$\rho$}} 
\newcommand{\Ggg}{\mbox{$\Gamma$}}
\newcommand{\ttt}{\mbox{$\tau$}} 
\newcommand{\bdd}{\mbox{$\partial$}}
\newcommand{\zzz}{\mbox{$\zeta$}}
\newcommand{\Ss}{\mbox{$\Sigma$}}
\newcommand{\Ddd}{\mbox{$\Delta$}}
\newcommand{\aub} {\mbox{$A \cup_{P} B$}}
\newcommand{\xuy} {\mbox{$X \cup_{Q} Y$}}
\newcommand{\Lll}{\mbox{$\Lambda$}}
\newcommand{\inter}{\mbox{${\rm int}$}}

\title{Berge's distance $3$ pairs of genus $2$ Heegaard splittings}  
\author{Martin Scharlemann} 
\date{\today}

\thanks{Research partially supported by an NSF grant. 
I am very grateful to John Berge for his discovery of these examples, for his help in understanding the examples, and for his very useful comments on earlier drafts of this paper.}

\begin{abstract}  Following an example discovered by John Berge \cite{Be2}, we show that there is a $4$-component link $L \subset (S^1 \times S^2)\#(S^1 \times S^2)$ so that, generically, the result of Dehn surgery on $L$ is a $3$-manifold with two inequivalent genus $2$ Heegaard splittings, and each of these Heegaard splittings is of Hempel distance $3$.  
\end{abstract}

\maketitle

\section{Introduction}
In \cite{Be2} John Berge introduces a criterion which, if satisfied by a genus two Heegaard splitting of a $3$-manifold, ensures that the Heegaard splitting is of distance $3$ or greater.  Furthermore, he gives an example of a pair of such Heegaard splittings of the same $3$-manifold, splittings which he shows to be inequivalent.   Such an example demonstrates conclusively that the list in \cite{RS} of possible manifolds with two or more inequivalent genus $2$ Heegaard splittings is incomplete.

Inspired by Berge's example, here we give a somewhat different view of his distance $3$ criterion (Sections \ref{sect:criterion} and \ref{sect:SUMS}) and then, after some preliminary discussion (Sections \ref{sect:standtrip} and \ref{sect:digress}), describe a general way of constructing further examples of the phenomenon he found (Sections \ref{sect:construct} and \ref{sect:afterword}).  All the examples are so-called Dehn derived pairs of splittings, specifically those of type $M_H$ described in \cite[Section 3]{BS}.  What we show here is that in a specific infinite family of Dehn-derived pairs of splittings, generic examples satisfy Berge's criterion and so are of Hempel distance $3$.  

\section{Berge's criterion for Heegaard distance $3$} \label{sect:criterion}

The goal of this section is to describe a criterion (see \cite{Be2}) which guarantees that a genus $2$ Heegaard splitting of a closed orientable $3$-manifold has Hempel distance at least $3$.  

Suppose that $H$ is a genus $2$ handlebody with oriented boundary $F$ and $\{ A, B \}$ is a complete set of meridian disks for $H$.  Let $\Gamma \subset F$ be the boundary of a third meridian disk for $H$ that separates $H$ into two solid tori $H_A$ and $H_B$ containing $A$ and $B$ respectively.  $\Gamma$ separates the surface $F$ into two punctured tori, $F_A \subset \bdd H_A$ and $F_B \subset \bdd H_B$.  

\begin{lemma}[Be2, Theorem 4.2, Claim 1]  \label{lemma:bergeclaim1}
Suppose $C$ is a simple closed curve in $F$ with this property: some arc in $F_A \cap C$ (resp $F_B \cap C$) intersects $\bdd A$ (resp $\bdd B$) at least twice. Then for any  meridian disk $D$ of $H$, $D \cap C \neq \emptyset$.
\end{lemma}

\begin{proof}  First we observe that there are no simple closed curves in the $4$-punctured sphere $F - (A \cup B)$ that are disjoint from $C$.  Indeed, $F - (A \cup B)$ is the union along $\Gamma$ of the $3$-punctured spheres (i. e. pairs of pants) $P_A = F_A - A$ and $P_B = F_B - B$.  The hypothesis implies that in $C \cap P_A$ (resp $C \cap P_B$) there is a subarc of $C$ with an end at each copy of $\bdd A$ (resp $\bdd B$).  But the only simple closed curve in the $4$-punctured sphere $F - (A \cup B)$ that is disjoint from these arcs is $\Gamma$.   $\Gamma$ separates $A$ from $B$, and $C$ intersects both $A$ and $B$, so $\Gamma \cap C \neq \emptyset$.  

Now suppose that $D$ is a meridian disk for $H$.  We have just observed that if $\bdd D$ is disjoint from $A \cup B$ then it intersects $C$.  So suppose $\bdd D$ intersects $A, B$ and let $D_0$ be an outermost disk of $D$ cut off by $A, B$.  Then $\bdd D_0$ intersects $F - (A \cup B)$ in an arc $\alpha$ with both ends at a single copy of, say, $\bdd A$ in $\bdd(F - (A \cup B))$.  $\alpha$ can't lie entirely in $P_A$, else $C$ would intersect one copy of $\bdd A$ more than it does the other.  But any arc of $\alpha \cap P_B$ has both ends on $\Gamma$ and so intersects the arc of $C \cap P_B$ identified above that has one end on each copy of $B$.
\end{proof}

Let $\lambda_a \subset F_A$ (resp $\lambda_b \subset F_B$) be simple closed curves which intersect $\bdd A$ (resp $\bdd B$) in exactly one point.  Orient $A, B, \lambda_a, \lambda_b$ so that the pairs of orientations $\lambda_a, \bdd A$ and $\lambda_b, \bdd B$ induce the given orientation on $F$ at the points of intersection.  

Let $\alpha$ be a properly embedded essential arc or simple closed curve in $F_A$, say.  Define $\rho_A(\alpha) \in \mathbb{Q} \cup \infty$ by $$\rho_A(\alpha) = \frac{\alpha \cdot \lambda_a}{\alpha \cdot \bdd A}.$$  Different choices of $\lambda_a$ will alter the numerator by multiples of the denominator, so, given $\alpha$, as long as $\rho_A(\alpha) < \infty$, we can choose $\lambda_a$ so $0 \leq \rho_A(\alpha) < 1$.  An alternate view of $\rho_A(\alpha)$ is useful:  lift $F_A$ to its universal abelian cover $\mathbb{R}^2 - \mathbb{Z}^2$ so that $\bdd A$ is vertical and $\lambda_a$ is horizontal.  Then $\rho_A(\alpha)$ is the slope of any lift of $\alpha$ to  $\mathbb{R}^2 - \mathbb{Z}^2$.  Similarly define $\rho_B$ for properly embedded essential arcs in $F_B$.

For any pair of rational numbers, $\frac pq, \frac rs$ let $$ |\frac pq, \frac rs| = |ps - qr|,$$ the absolute value of the determinant of the associated $2 \times 2$ matrix. Then $\rho_A$ has the following pleasant property:  Suppose $\alpha, \beta$ are both properly embedded arcs in $F_A$ isotoped to intersect minimally.  Then either $\alpha$ and $\beta$ are parallel and $\rho_A(\alpha) = \rho_A(\beta)$ or the arcs intersect in exactly $|\rho_A(\alpha), \rho_A(\beta)| - 1$ points.  In particular, $\rho_A$ can take on at most three distinct values for arcs appearing in any given collection $C$ of disjoint essential arcs in $F_A$. (A significant example would be slopes $0, 1, \infty$.  Moreover, the set of denominators (all of which we may take to be non-negative) that appear among the values of $\rho_A$ do not depend on the choice of $\lambda_A$; they represent the number of times each choice of arc in $C$ intersects $\bdd A$.  This motivates the following definition:

\begin{defin} \label{defin:denom}  Suppose $C$ is a disjoint collection of simple closed curves in $F$ that essentially intersects $\Gamma$.  Let $denom_A(C) \subset \mathbb{N}_+$ be the set of denominators that appear among arcs in $C \cap F_A$.  Similarly, let $denom_B(C) \subset \mathbb{N}_+$ be the set of denominators that appear among arcs in $C \cap F_B$.  
\end{defin}

\begin{defin}  \label{defin:hdenom} Suppose $C \subset F$ is a collection of disjoint simple closed curves in $F$.  Then $C$ has {\em high denominators} with respect to $A$ (resp. $B$) if there are $r, s \in denom_A(C)$ (resp $denom_B(C)$ so that $2 \leq r \leq s - 2$.  
\end{defin}

\begin{prop}  \label{prop:hmeansself} Suppose $C \subset F$ is a set  of disjoint simple closed curves in $F$ that has high denominators with respect to $A$ (resp $B$).  Then any simple closed curve $C'$ that is disjoint from $C$ has the property that each component of $C' \cap F_A$ (resp $C' \cap F_B$) intersects $\bdd A$ (resp $\bdd B$) at least twice.  
\end{prop}

\begin{proof}  The alternative is that there is an arc $\alpha$ among the components of $C' \cap F$ with denominator $0$ or $1$ and so slope $\frac 10$ or $\frac 01$.  But for such an arc to be disjoint from the arc $\beta_r$ (resp $\beta_s$) with denominator $r \geq 2$ (resp $s \geq 4$), $\frac 10$ is impossible and the slope of $\beta_r$ must then be $\frac {\pm 1}r$ and the slope of $\beta_s$ must be $\frac {\pm 1}s$.  But this makes $$|\rho_A(\beta_r), \rho_A(\beta_s)| = |\pm r \pm s| > 1,$$ contradicting the fact that $\beta_r$ and $\beta_s$ are disjoint.
\end{proof}

\begin{defin}  \label{defin:dhdenom} Suppose $C_1, C_2 \subset F$ is a pair of non-parallel, non-separating disjoint simple closed curves in $F$.  For each $i = 1, 2$, a $C_i$-rectangle in $F_A$ (resp $F_B$) is a pair of parallel arcs of $C_i \cap F_A$ (resp $C_i \cap F_B$) so that the region between the parallel arcs is disjoint from $C_1 \cup C_2$.  

The pair $C_1, C_2$ satisfies the {\em high denominator rectangle condition} with respect to $A$ (resp. $B$) if, for each $i = 1, 2$,
there are at least two $C_i$-rectangles in $F_A$ (resp. $F_B$), one of denominator $r_i \geq 2$ and one of denominator $s_i \geq r_i + 2$.  

The pair $C_1, C_2$ satisfies the {\em high denominator rectangle condition} in $F$ if it satisfies the high denominator rectangle condition with respect to both $A$ and $B$.
\end{defin}

Notice that an argument that $C_1, C_2$ satisfies the high denominator rectangle condition in $F$ will involve a minimum of $16$ arcs:  two arcs in each rectangle, and two rectangles from each of $C_1, C_2$ in each of $F_A$ and $F_B$.  

\begin{prop}  \label{prop:dhmeansh} Suppose $H \cup_F J$ is a Heegaard splitting and $J$ has a complete pair of meridian disks $X, Y$ so that the pair $\bdd X, \bdd Y$ satisfies the high denominator rectangle condition with respect to $A$.   Suppose $E$ is any non-separating meridian disk for $J$.  Then $\bdd E$ has high denominators with respect to $A$.
\end{prop}

\begin{proof}   Consider a rectangle $R_X$ in $F_A$ cut off from $F_A$ by two parallel arcs of $\bdd X$.  Since $R_X$ is disjoint from $\bdd Y$ the interior of the rectangle lies in the $4$-punctured sphere $F_{X, Y} = F - (X \cup Y)$, appearing there as cut off by a pair of parallel arcs of $\Gamma \cap F_{X, Y}$.  The ends of these arcs lie in either the same or separate copies of $\bdd X$ in $\bdd F_{X, Y}$.  The same could be said for a rectangle $R_Y$ in $F_A$ cut off by parallel arcs of $\bdd Y$, and the only way that both can be true is if, in both cases, the arcs of $\Gamma \cap F_{X, Y}$ connect one copy of $\bdd X$ (resp $\bdd Y$) to the other copy of $\bdd X$ (resp $\bdd Y$).    

If $E = X$ or $Y$ the proposition is immediate.  If $E$ is disjoint from both $X$ and $Y$ then, since $E$ is non-separating, $\bdd E$ must pass through both rectangles $R_X$ and $R_Y$; a subarc of $\bdd E$ that passes through $R_X$ will appear in $F_A$ as an arc parallel to the sides of $R_X$ in $F_A$ and so will have the same denominator.  Thus a denominator for the rectangle $R_X$ in $F_A$ appears also as denominator for $\bdd E$.  This shows that $\bdd E$ satisfies the high-denominator condition.

The picture is only slightly different if $E$ is not disjoint from $X \cup Y$.  An outermost disk of $E$ cut off by an arc of $(X \cup Y) \cap E$ will intersect $F_{X, Y}$ in an arc $\alpha$ with both ends at the same copy of $Y$, say, and will separate one copy of $\bdd X$ in $\bdd F_{X, Y}$ from the other.  In particular, $\alpha$ will pass through each rectangle that runs between the two copies of $\bdd X$ in $\bdd  F_{X, Y}$.  Once again, the denominator for $R_X$ in $F_A$ appears also as a denominator for $\bdd E$. \end{proof}

\begin{thm}  \label{thm:main}  For a  genus $2$ Heegaard splitting $H \cup_F J$, suppose there is a complete pair of meridian disks $\{ A, B \}$ in $H$, separated by a meridian disk $\Gamma \subset H$, and there is a complete pair of meridian disks $\{ X, Y \}$ in $J$, so that the pair $\bdd X, \bdd Y$ satisfies the high denominator rectangle condition in $F$.   Then the Heegaard splitting has Hempel distance at least $3$.
\end{thm}

\begin{proof}  We show that the splitting does not have the disjoint curve property \cite{Th}.  Let $C$ be any simple closed curve in $F$, and $E$ be any meridian disk of $J$ that is disjoint from $C$.  If $E$ is separating, then $C$ lies on one side of $E$ in $J$, so we can replace $E$ by disk in $J$ disjoint from $C$ that is also non-separating.  So, with no loss of generality, assume $E$ is non-separating.

By Proposition \ref{prop:dhmeansh} $\bdd E$ has high denominators with respect to both $A$ and $B$.  By Proposition \ref{prop:hmeansself}, each component of $C \cap F_A$ (resp $C \cap F_B$) intersects $\bdd A$  (resp $\bdd B$) at least twice.  It then follows from Lemma \ref{lemma:bergeclaim1} that any meridian $D$ of $H$ intersects $C$.
\end{proof}

\section{High denominator rectangles and SUMS} \label{sect:SUMS}

Suppose $H \cup_F J$ is a genus $2$ Heegaard splitting and $\{ A, B \}$ is a complete set of meridian disks for $H$.  In \cite{Be2} Berge calls the pair $\{ A, B \}$ a set of Strict Universal Minimizers (SUMS) for the Heegaard splitting if for any complete pair of meridian disks $\{ X, Y \}$ for $J$ and alternate complete pair of meridian disks $\{ A^*, B^* \}$ for $H$, $|(A \cup B) \cap (X \cap Y)| < |(A^* \cup B^*) \cap (X \cap Y)|$.   Here is his argument that if $J$ has any complete pair of meridian disks $\{ X, Y \}$ so that the pair $\bdd X, \bdd Y$ satisfies the high denominator rectangle condition with respect to both $A$ and $B$, then the pair $\{ A, B \}$ is a set of SUMS.

\begin{lemma} \label{lemma:altAB1}
Suppose $C$ is a collection of disjoint simple closed curves in $F$ that essentially intersects $\Gamma \subset F$ and every arc in $F_A \cap C$ (resp $F_B \cap C$) intersects $\bdd A$ (resp $\bdd B$) at least twice. Then for any  non-separating meridian disk $D$ of $H$ disjoint from $A$ and $B$, either $D$ is parallel to $A$ or $B$ or $|D \cap C| \geq |A \cap C|$ and $|D \cap C| \geq |B \cap C|$.

If in addition some arc in $F_A \cap C$ (resp $F_B \cap C$) intersects $\bdd A$ (resp $\bdd B$) more than twice,  either $D$ is parallel to $A$ or $B$ or the inequalities are strict: $|D \cap C| > |B \cap C|$ (resp $|D \cap C| > |A \cap C|$).
\end{lemma}

\begin{proof} Choose $D$ to be a non-separating meridian disk, disjoint from $A$ and $B$, not parallel to $A$ or $B$, so that among all such disks $|D \cap C|$ is minimal.  

$\bdd D$ is non-separating so it can't be parallel to $\Gamma$.  This implies that $\bdd D$ intersects both pairs of pants $P_A$ and $P_B$ in essential arcs, each with both ends on $\Gamma$.  Let $\bdd A_+, \bdd A_-$ be the two copies of $\bdd A$ in $\bdd P_A$.  Let $p$ be the number of arcs of $C \cap P_A$ that have an end on each of $\bdd A_{\pm}$ and $q$ be the number of arcs of $C \cap P_A$ that have one end on $\bdd A_+$ and one end on $\Gamma$.  Then $|A \cap C| = |\bdd A_+ \cap C| = p + q$ and, counting also the arcs running from $\bdd A_-$ to $\Gamma$, $\Gamma \cap C = 2q$.  The hypothesis guarantees that each arc of $\bdd D \cap P_B$ intersects each arc of $C \cap F_B$ so $\bdd D$ intersects $C \cap F_B$ at least $q$ times.  Similarly, each arc of $\bdd D \cap P_A$ intersects $C$ at least $p$ times.  Hence $|D \cap C| \geq p + q = |A \cap C|$.  A symmetric argument shows $|D \cap C| \geq |B \cap C|$.  

Continuing in the same vein, if in addition some arc in $F_B \cap C$ intersects $\bdd B$ more than twice then $\bdd D$ will intersect that arc more than once, so $\bdd D$ will intersect $C \cap F_B$ more than $q$ times and the inequality becomes strict:  $|D \cap C| > p + q = |A \cap C|$. 

\end{proof}

\begin{lemma} \label{lemma:altAB2}
Suppose $C$ is a collection of disjoint simple closed curves in $F$ that essentially intersects $\Gamma \subset F$ and every arc in $F_A \cap C$ (resp $F_B \cap C$) intersects $\bdd A$ (resp $\bdd B$) at least twice. Then for any  non-separating meridian disk $D$ of $H$ disjoint from $A$ and not parallel to $A$ or $B$, $|D \cap C| \geq |A \cap C|$ and $|D \cap C| \geq |B \cap C|$.

If in addition some arc in $F_A \cap C$ (resp $F_B \cap C$) intersects $\bdd A$ (resp $\bdd B$) more than twice,  either $D$ is parallel to $A$ or $B$ or the inequalities are strict: $|D \cap C| > |B \cap C|$ (resp $|D \cap C| > |A \cap C|$).
\end{lemma}

\begin{proof} Choose $D$ to be a non-separating meridian disk, disjoint from $A$ and not parallel to $A$ or $B$, so that among all such disks $|D \cap C|$ is minimal.  Then choose, among all such disks, one whose intersection 
with $B$ is minimal.

Following Lemma \ref{lemma:altAB1} we need only consider the case in which $\bdd D$ intersects $B$; we will show that this is impossible.  An outermost disk $E$ of $D$ cut off by $B$ necessarily intersects $F$ in an arc of $F 
- (A \cup B)$ with  both ends at a copy of $\bdd B$ in $\bdd F  - (A \cup B)$. It follows as above that the 
arc intersects $P_A$ in one or more essential arcs with both ends at $\Gamma$. In 
particular it crosses all $p$ arcs of $C \cap P_A$ that have an end on each of $\bdd A_{\pm}$. 
Let $D'$ be the meridian disk for $H$ obtained by band-summing $D$ to $A$ 
via a band parallel to the $p$ arcs.  This can be done so that the new arc of $D' - B$ is parallel in $F$ to a subarc of $\bdd B$ with the same ends, so $|D' \cap B| < |D \cap B|$.  We also have
$$|D' \cap C | \leq |D \cap C | +|A \cap C | -2p = 
|D \cap C | - (p - q).$$ Now the hypothesis guarantees that $C$ intersects 
$A$ at least as often as $\Gamma$ does, so $p + q 
\geq 2q \implies p - q \geq 0$. Hence 
$|D' \cap C | \leq |D \cap C |$. In a genus two surface, the band sum of 
two non-separating curves cannot be separating unless the original two 
curves are parallel, so it follows that $D'$ is non-separating. Since one 
of the band-summands is $A$, $D'$ is not parallel to $A$. If $D'$ were parallel 
to $B$ then, dually, $D$ could be obtained by band-summing $A$ to $B$ and 
so would have been disjoint from $B$. We conclude that $D'$ satisÞes the 
hypothesis of the theorem, yet $|D' \cap C| \leq |D \cap C|$ and $|D' \cap B| < |D \cap B|$.  This contradicts our choice of $D$.
\end{proof}

\begin{prop}[Be2, Lemma 4.4]  \label{prop:Lemma4_4} Suppose $C$ is a collection of disjoint simple closed curves in $F$ that essentially intersects $\Gamma$ and every arc in $F_A \cap C$ (resp $F_B \cap C$) intersects $\bdd A$ (resp $\bdd B$) at least twice. (That is, every such arc has denominator at least $2$.)  Then for any  complete set of meridian disks $A^*, B^*$ not parallel to $A, B$, $|(A^* \cup B^*) \cap C| \geq |(A \cup B) \cap C|$.

If in addition the denominator of some arc of $F_A \cap C$ and some arc of $F_B \cap C$ are each at least $3$, then for any  complete set of meridian disks $A^*, B^*$ not parallel to $A, B$, $|(A^* \cup B^*) \cap C| > |(A \cup B) \cap C|$.
\end{prop}

\begin{proof} Choose a complete set of meridian disks $A^*, B^*$ for $H$ not parallel to $A, B$ so that among all such disks $|(A^* \cup B^*) \cap C|$ is minimal.  Then choose, among all such pairs of disks, a pair for which $|(A^* \cup B^*) \cap (A \cup B)|$ is minimal. 

If $A^* \cup B^*$ is disjoint from $A \cup B$ then the result follows from Lemma \ref{lemma:altAB2}, so suppose $A^* \cup B^*$ intersects $A \cup B$.  Since $A^*, B^*$ have been chosen to minimize the number of intersection components, all of the intersection components are arcs.  Of the collection of outermost disks cut off from any of the disks $A^*, B^*, A, B$ by these arcs of intersection, let $E$ be one that intersects $C$ least.  

We first show that $E$ cannot lie in $A^*$ or $B^*$.  Say $E \subset A^*$ is cut off by an arc of $A \cap A^*$.  Then $E$ can be used to cut $A$ into two non-parallel disks, each disjoint from $A \cup B$ and each intersecting $C$ no more than $A$ does.  This contradicts Lemma \ref{lemma:altAB1}.

We next show that $E$ cannot lie in $A$ or $B$.  Say $E \subset A$ is cut off by an arc of $A \cap A^*$. Then $E$ can be used to cut $A^*$ into two non-parallel disks, $A^*_1, A^*_2$, each disjoint from $A^* \cup B^*$, each intersecting $C$ no more than $A^*$ and each intersecting $A \cup B$ in fewer arcs than $A^*$ did.  In a genus two handlebody the band sum of a non-separating and a separating disk is parallel to the non-separating disk; it follows that $A^*_1, A^*_2$ are both non-separating. Since they are not parallel, at least one, say $A^*_1$, is not parallel to $B^*$.  Hence the pair $A^*_1, B^*$ is a complete collection of meridian disks.  The pair cannot be parallel to $A, B$ since $A^*$ is a band sum done in the complement of the two and any band-sum of non-separating disks, all done in the complement of $A$ and $B$ is disjoint from $A \cup B$.  Then the pair $A^*_1, B^*$ contradicts our choice of $A^*, B^*$.  
\end{proof} 

\begin{cor}  \label{cor:SUMS} Suppose $H \cup_F J$ is a Heegaard splitting and $J$ has a complete pair of meridian disks $X, Y$ so that the pair $\bdd X, \bdd Y$ satisfies the high denominator rectangle condition.   Then the pair $A, B$ is a set of SUMS for the splitting.
\end{cor}

\begin{proof}  Let $X^*, Y^*$ be any complete pair of meridian disks for $J$.  Then according to Proposition \ref{prop:dhmeansh}, both $\bdd X^*$ and $ \bdd Y^*$ have high denominators with respect to both $A$ and $B$.  Following Proposition \ref{prop:hmeansself}, every component of $(\bdd X^* \cup \bdd Y^*) \cap F_A$ (resp $(\bdd X^* \cup \bdd Y^*) \cap F_B$ intersects $\bdd A$ (resp $\bdd B$) at least twice, and those arcs with highest denominators intersect $\bdd A$ (resp $\bdd B$) more than twice.  The result then follows from Proposition \ref{prop:Lemma4_4}.
\end{proof}

\section{Construction preliminaries: the standard triple} \label{sect:standtrip}

As in Section \ref{sect:criterion}, let $H$ be an oriented handlebody divided into solid tori $H_A, H_B$ by a separating meridian $\Gamma$, with meridians of the solid tori denoted  $A$ and $B$ respectively.  Suppose further we are given specified longitudes of the solid tori, denoted $\lambda_A \subset \bdd F_A$ and $\lambda_B \subset F_B$.  Given this datea, pictured in Figure \ref{fig:doubletorus} is a {\em standard triple} $\Gamma_A, \Gamma, \Gamma_B$ of separating meridians in $H$. The notation is chosen so that $\Gamma_A$ intersects $F_A$  in two meridians and is disjoint from $\lambda_B$ in $F_B$; symmetrically $\Gamma_B$ intersects $F_B$  in two meridians and is disjoint from $\lambda_A$ in $F_A$.

 A somewhat different view of a standard triple is given via Figure \ref{fig:standtrip}.  The boundary of $\Gamma_B$ is shown in blue; the $\mathbb{Z}_3$ symmetry given by $\frac{2\pi}{3}$ rotation about the vertical axis induces a cyclic permutation of the handles that carries $\Gamma_B$ first to $\Gamma_A$ and then to $\Gamma$.  It also carries $\lambda_A$ first to $\lambda_B$ and then to a third curve $\lambda_{AB}$ that is shown in red in Figure \ref{fig:standtrip}.  The curve $\lambda_{AB}$ intersects $\Gamma$ in two points and is disjoint from both $\Gamma_A$ and $\Gamma_B$.  

 \begin{figure}[ht!]
 \labellist
\small\hair 2pt
\pinlabel \color{red}{$\lambda_A$} at 90 145
\pinlabel \color{red}{$\lambda_B$} at 365 145
\pinlabel \color{blue}{$\Gamma_B$} at 45 145
\pinlabel \color{black}{$\Gamma_A$} at 400 145
\pinlabel $\Gamma$ at 220 270
\pinlabel $F_A$ at 70 270
\pinlabel $F_B$ at 375 270
  \endlabellist
    \centering
    \includegraphics[scale=0.5]{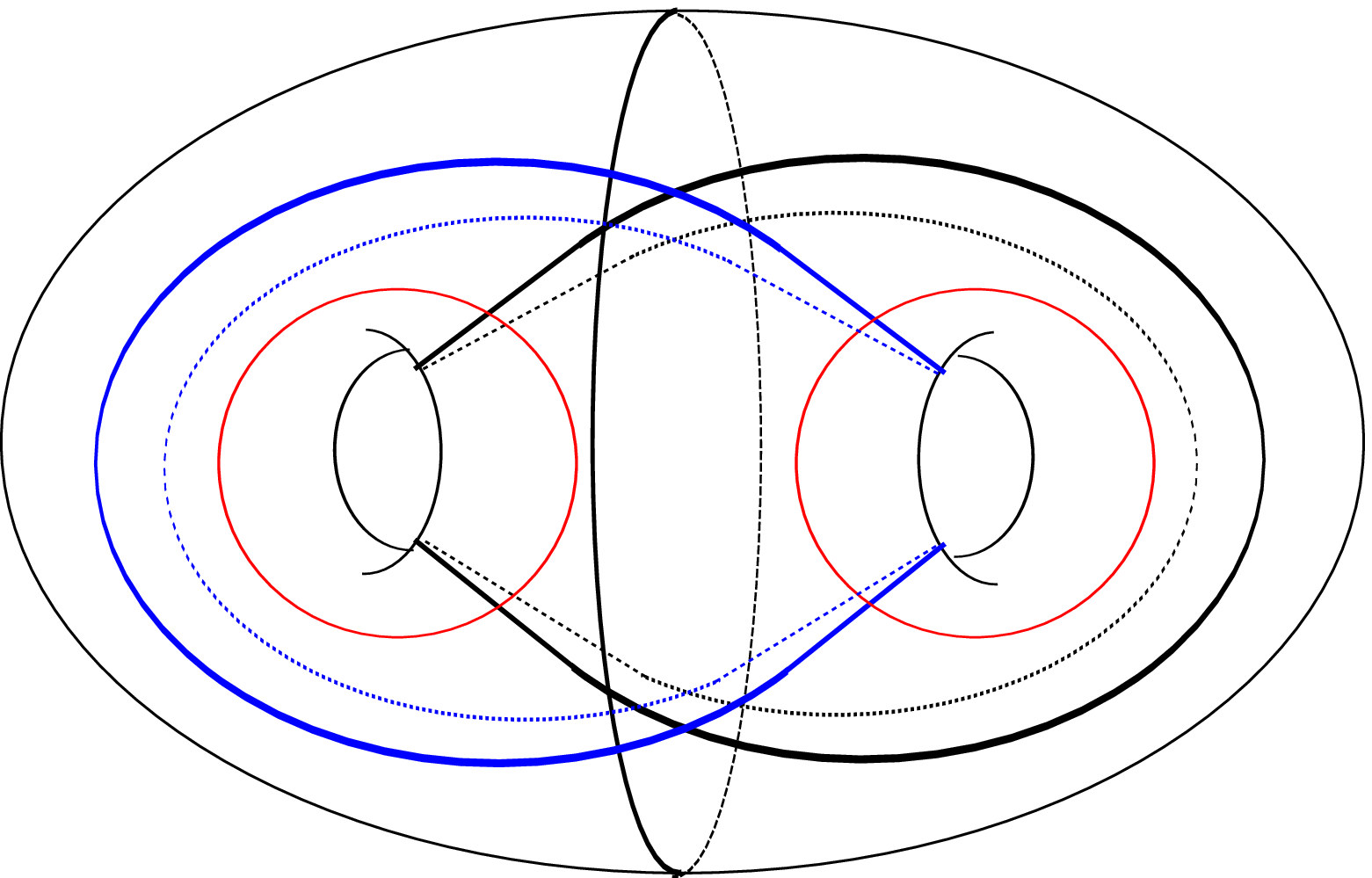}
    \caption{} \label{fig:doubletorus}
    \end{figure}
    
    \begin{figure}[ht!]
 \labellist
\small\hair 2pt

\pinlabel  \color{red}{$\lambda_{AB}$} at 145 165
\pinlabel \color{black}{$\Gamma_B$} at 55 260

  \endlabellist
    \centering
    \includegraphics[scale=0.7]{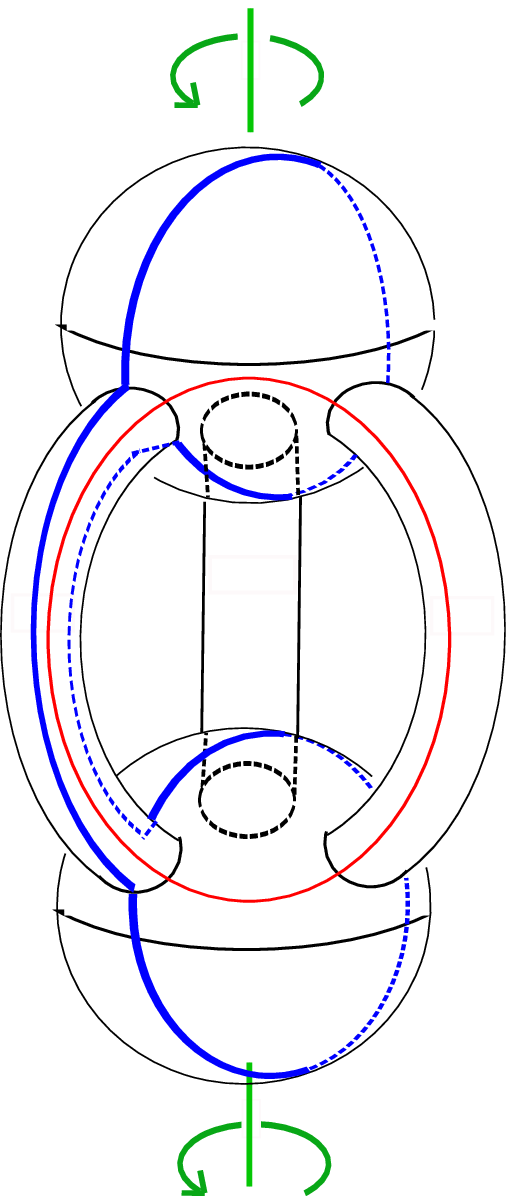}
\caption{} \label{fig:standtrip}
    \end{figure}

Suppose $D$ is a separating meridian disk for $H$, dividing $H$ into two solid tori $U$ and $V$.  Consider the automorphism $\tau_D: H \to H$ given by performing a left half-twist in a collar neighborhood of $D$ as one crosses from $U$ to $V$.  The automorphism is the identity on $U$ but, because of the half-twist, reverses the orientation of both the meridian and longitude of $V$.  If the labels $U$ and $V$ are reversed, the resulting automorphism differs from the original by the hyperelliptic involution.   In particular, for $c$ any simple closed curve in $F$ transverse to $\bdd D$, the image $\tau_D(c)$ is, up to isotopy, independent of which half is labeled $U$ and which $V$, so without ambiguity we denote $\tau_D(c)$ by $c^D$.  (But because the hyperelliptic involution may reverse the orientation of a curve, an orientation of the curve $c$ does not induce a natural orientation on the curve $c^D$.)

\begin{example}  \label{example:halftwist}
\rm{Figure \ref{fig:halftwist} shows the result of half-twisting $\bdd \Gamma_B$ along $\Gamma_A$ (simplified in the second drawing).   Note that the resulting curve $\bdd \Gamma_B^{\Gamma_A}$ intersects $F_A$ in two rectangles (just the front part of the rectangles are shown in Figure \ref{fig:rectangle}), with slopes $\infty$ and $-1$, and also intersects $F_B$ in two rectangles, with slopes $0$ and $+1$.}  
\end{example}

(Whether the slopes are $+1$ or $-1$ is relatively unimportant: the definitions in Section \ref{sect:digress} will be made highly symmetric, so here and elsewhere we can be relatively relaxed about whether specific slopes are positive or negative.)

    \begin{figure}[ht!]
 \labellist
\small\hair 2pt
\pinlabel \color{blue}{$\bdd \Gamma_B^{\Gamma_A}$} at 210 560
\pinlabel  \color{black}{${F_Q}$} at 155 250
\pinlabel ${F_P}$ at 155 215
 \endlabellist
    \centering
    \includegraphics[scale=0.5]{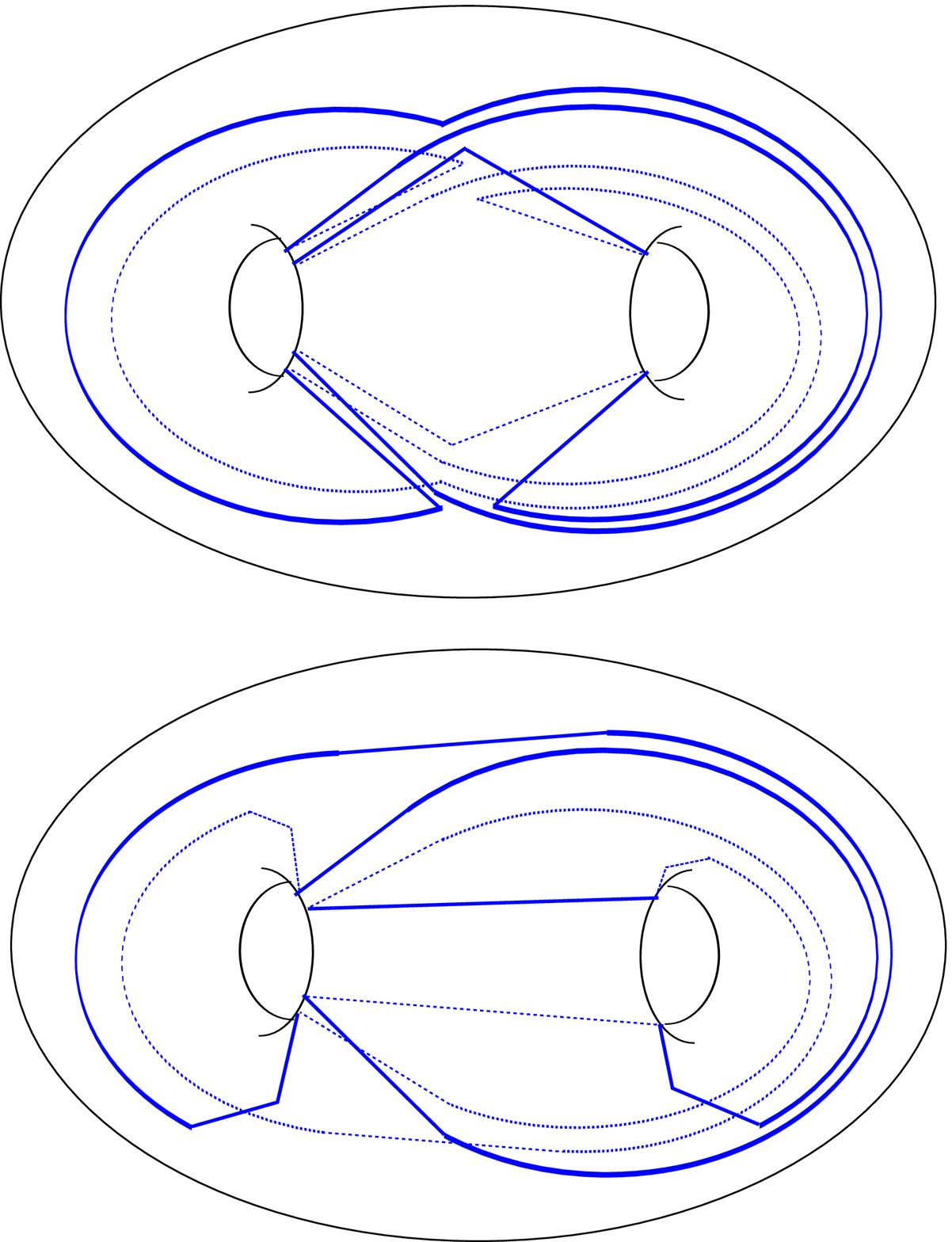}
    \caption{} \label{fig:halftwist}
    \end{figure}
    
        \begin{figure}[ht!]
 \labellist
\small\hair 2pt
  \endlabellist
    \centering
    \includegraphics[scale=0.5]{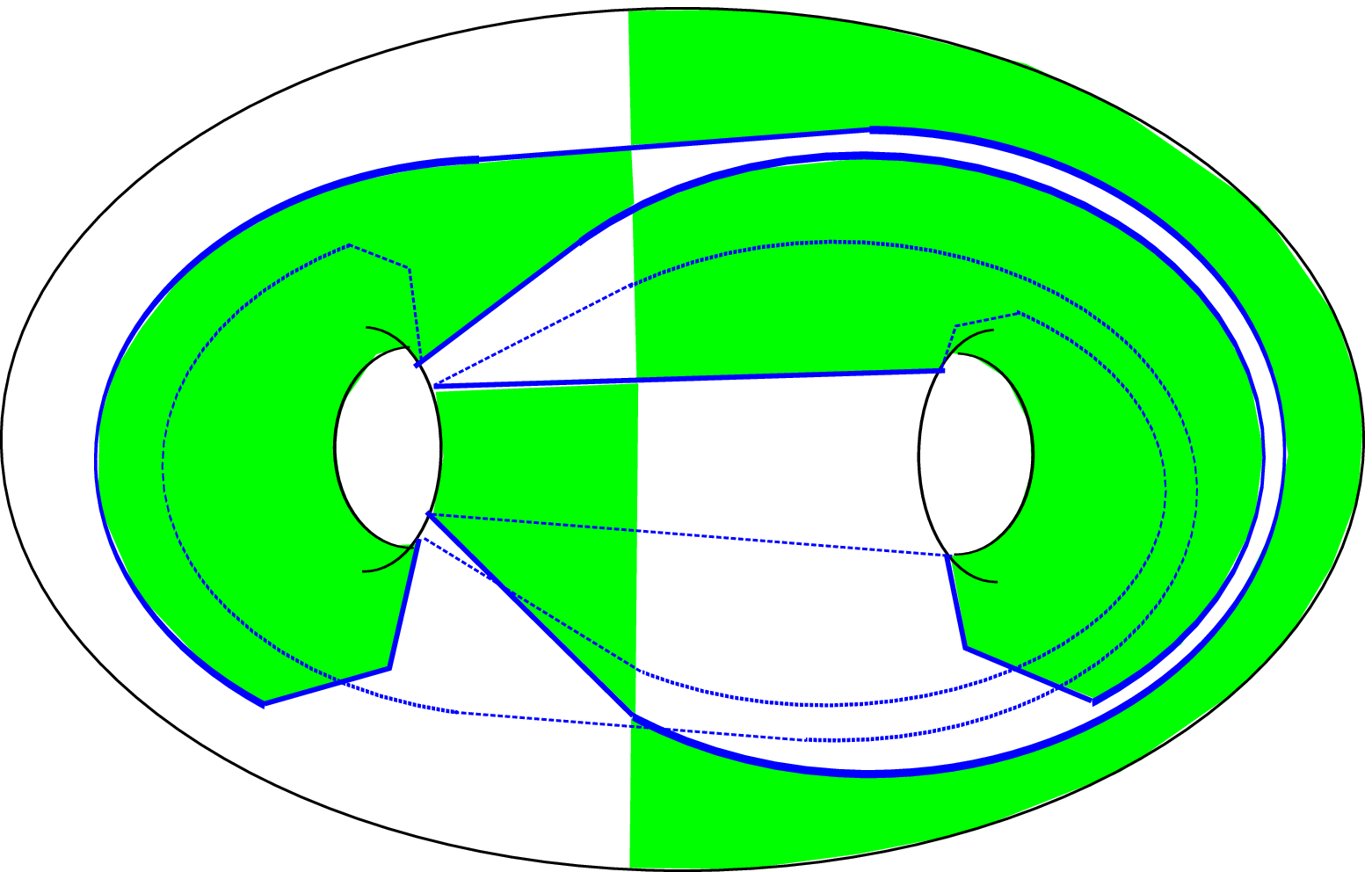}
    \caption{} \label{fig:rectangle}
    \end{figure}

\section{Digression on rationals and slopes} \label{sect:digress}

    \begin{figure}[ht!]
 \labellist
\small\hair 2pt
\pinlabel \color{blue}{$2/3^*$} at 170 240
\pinlabel \color{blue}{$-2/3^*$} at 170 0
\pinlabel \color{blue}{$3/2^*$} at 70 240
\pinlabel \color{blue}{$-3/2^*$} at 70 0
\pinlabel \color{blue}{$1/3^*$} at 250 170
\pinlabel \color{blue}{$-1/3^*$} at 250 75
\pinlabel \color{blue}{$3/1^*$} at -5 170
\pinlabel \color{blue}{$-3/1^*$} at -8 75
\pinlabel \color{black}{$1/1$} at 120 245
\pinlabel \color{black}{$-1/1$} at 120 -8
\pinlabel \color{black}{$1/2$} at 215 215
\pinlabel \color{black}{$-1/2$} at 225 35
\pinlabel \color{black}{$2/1$} at 20 215
\pinlabel \color{black}{$-2/1$} at 10 35
\pinlabel \color{black}{$0/1$} at 255 120
\pinlabel \color{black}{$1/0$} at -15 120
 \endlabellist
    \centering
    \includegraphics[scale=0.8]{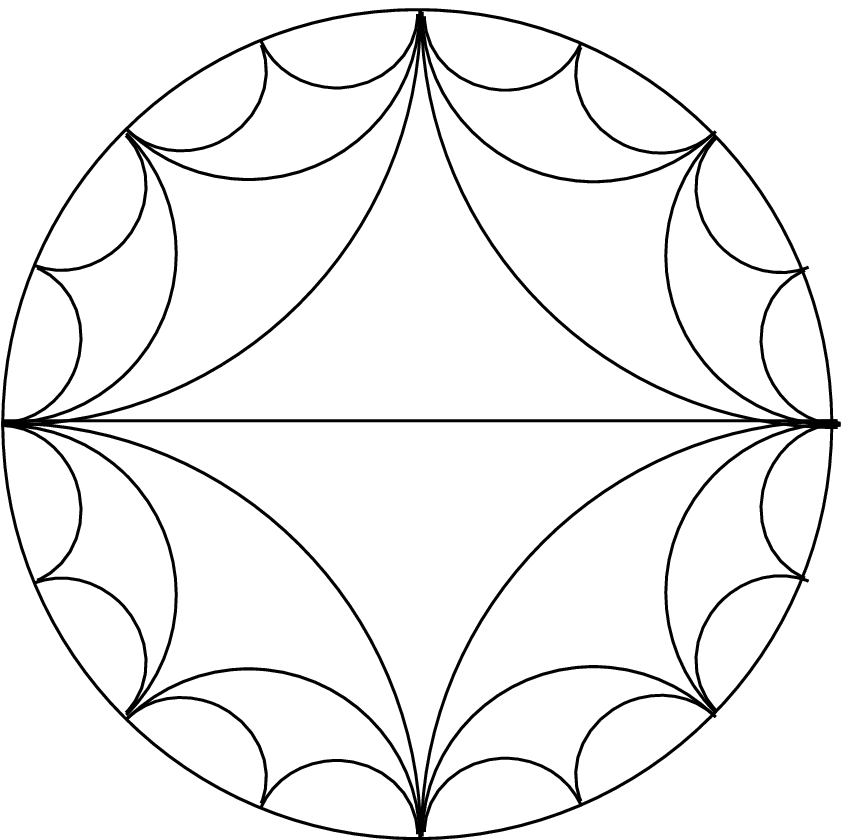}
    \caption{} \label{fig:Farey}
    \end{figure}
    
 \begin{defin} \label{defin:close}  The rationals $\infty, 0, \pm 1, \pm 2, \pm \frac 12$ (those that are black and unstarred in the part of the Farey tesselation shown in Figure \ref{fig:Farey}) are called {\em close} rationals; the close rationals together with $\pm 3, \pm \frac 13, \pm \frac 23, \frac 32$ (those that are blue and starred in Figure \ref{fig:Farey}) are called {\em nearby} rationals.
    \end{defin}

\begin{defin} \label{defin:distant} A rational $\frac pq$ is {\em distant} if $|p|, |q|, |p \pm q|, |p \pm  2q|$ and $|2p \pm q|$ are all $\geq 2$.  

A distant rational $\frac pq$ is {\em remote} if, in addition, $|p \pm 3q|, |3p \pm  2q|, |2p \pm  3q|$ and $|3p \pm q|$ are all $\geq 2$. 
\end{defin}

There are several equivalent ways of saying much the same thing; the first two motivate the terminology:
\begin{itemize}
\item  A rational is distant if and only if it is not adjacent in the Farey tesselation (Figure \ref{fig:Farey}) to a close rational.
\item  A rational is remote if and only if it is not adjacent in the Farey tesselation to a nearby rational.
\item The rational $\frac pq$ is remote if and only $\frac pq, \frac {p \pm q}q$ and $\frac p{q \pm p}$ are all distant.
\end{itemize}

Simply because there are more large numbers than small numbers, a ``random" rational will satisfy all these conditions and so will be remote. 

\bigskip

Let $W$ be a solid torus with a specified longitude and $c$ be a simple closed curve on $\bdd W$. 

\begin{defin} \label{defin:genericslope}  $c$ has {\em distant (resp. remote, close, nearby)} slope if  the rational that represents the slope of $c$ is distant (resp. remote, close, nearby).  
\end{defin}

Put another way, $c$ has distant slope if the number of intersection points of $c$ with each of the simple closed curves of slopes $0, \pm 1, \pm 2, \pm \frac 12, \infty$ is at least $2$.  The curve $c$ furthermore has remote slope if also the number of intersection points with each of the simple closed curves of slopes $\pm \frac 13, \pm 3, \pm \frac 23,  \pm \frac 32$ is at least $2$.   

Let $\bdd W_- \subset \bdd W$ be a punctured torus obtained from $\bdd W$ by removing the interior of a disk in the complement of $c$.  

\begin{prop} Suppose $c$ has remote slope, and $\alpha, \beta$ are disjoint properly embedded non-parallel arcs in $\bdd W_-$ with close slopes.   Then $|c \cdot \alpha|, |c \cdot \beta|$ and $||c \cdot \alpha|-|c \cdot \beta||$ are each at least $2$.
\end{prop}

\begin{proof}  The definition of close slope is so symmetric that we may as well assume that $0 \leq c \cdot \alpha \leq c \cdot \beta$.  Since the arcs $\alpha$ and $\beta$ are disjoint in $\bdd W_-$ the slopes that they represent in $\bdd W$ are adjacent in the Farey tesselation (Figure \ref{fig:Farey}) and the assumption is that the slopes are both given there by the unstarred numbers.  It is easy to see that the number $c \cdot \beta - c \cdot \alpha$ is the same as $c \cdot \gamma$, where $\gamma$ is a simple closed curve whose slope is adjacent in the Farey tesselation to the slopes of both $\alpha$ and $\beta$, which puts $\gamma$ among the nearby slopes, those that appear, either starred or unstarred, in Figure \ref{fig:Farey}.  But then, since $c$ is remote,  $c \cdot \alpha, c \cdot \beta$ and $c \cdot \gamma$ are all at least $2$. 
\end{proof}

\begin{cor} \label{cor:surgintersect} Suppose $W^{surg}$ is the solid torus obtained from $W$ by Dehn surgery on the core of $W$, with a surgery coefficient that is a remote rational.  Suppose $\alpha, \beta$ are disjoint properly embedded non-parallel arcs in $\bdd W_-$ with close slopes. Then $\alpha, \beta$ have high denominators (see \ref{defin:hdenom}) with respect to the meridian of $W^{surg}$.  
\end{cor}

In a similar vein, it will be useful to have this combinatorial lemma, which will eventually be applied in each punctured torus $F_A, F_B$:

\begin{lemma} \label{lemma:octagon}  Suppose $\alpha$ and $\beta$ are non-parallel disjoint essential proper arcs in a punctured torus $T$.  Suppose $c$ is a finite set of disjoint essential simple closed curves in $T$ isotoped to minimally intersect $\alpha, \beta$ so that $|\alpha \cap c| = p < q =|\beta \cap c|$.   Let $P$ denote the octagon $T - \eta(\alpha \cup \beta)$, with four sides coming from $\bdd T$ and two sides coming from each of $\alpha, \beta$.  Then the arcs $c \cap P$ consist of $2p$ arcs that are each parallel to one the $4$ sides in $P$ coming from $\bdd T$, and $q - p$ parallel arcs, each essential in $T$, that are not parallel to any side of $P$.
\end{lemma}

\begin{proof}  This is an elementary counting argument in the octagon $P$, see Figure \ref{fig:octagon}.  There cannot simultaneously be an arc in $P$ connecting both copies of $\alpha$ in $\bdd P$ and an arc in $P$ connecting both copies of $\beta$ in $\bdd P$, since such arcs cross.  Since $q > p$ it follows that all arcs with one end in a copy of $\alpha$ must have other end in a copy of $\beta$ so there are $2p$ of these.  The remaining arcs must connect one copy of $\beta$ to the other, using up the other $2q - 2p$ endpoints.  
\end{proof}

In Figure \ref{fig:octagon}, the four arcs of $c$ immediately parallel to the four $\bdd T$-sides of $P$ together constitute a component of $c$ that is parallel to $\bdd T$.  So, if $c$ is a single non-separating curve in $T$ (as it will be in the application), then two opposite $\bdd T$-sides of $P$ will not be parallel to any arcs of $c \cap P$ and the other pair of $\bdd T$-sides of $P$ will each be parallel to $p$ arcs of $c \cap P$.  

    \begin{figure}[ht!]
 \labellist
\small\hair 2pt
\pinlabel \color{blue}{$\bdd T$} at 25 290
\pinlabel \color{blue}{$\bdd T$} at 300 290
\pinlabel \color{blue}{$\bdd T$} at 25 25
\pinlabel \color{blue}{$\bdd T$} at 300 25
\pinlabel \color{red}{$c$} at 150 150
\pinlabel \color{black}{$\beta$} at -20 155
\pinlabel $\beta$ at 355 155
\pinlabel $\alpha$ at 170 340
\pinlabel $\alpha$ at 170 -10

 \endlabellist
    \centering
    \includegraphics[scale=0.5]{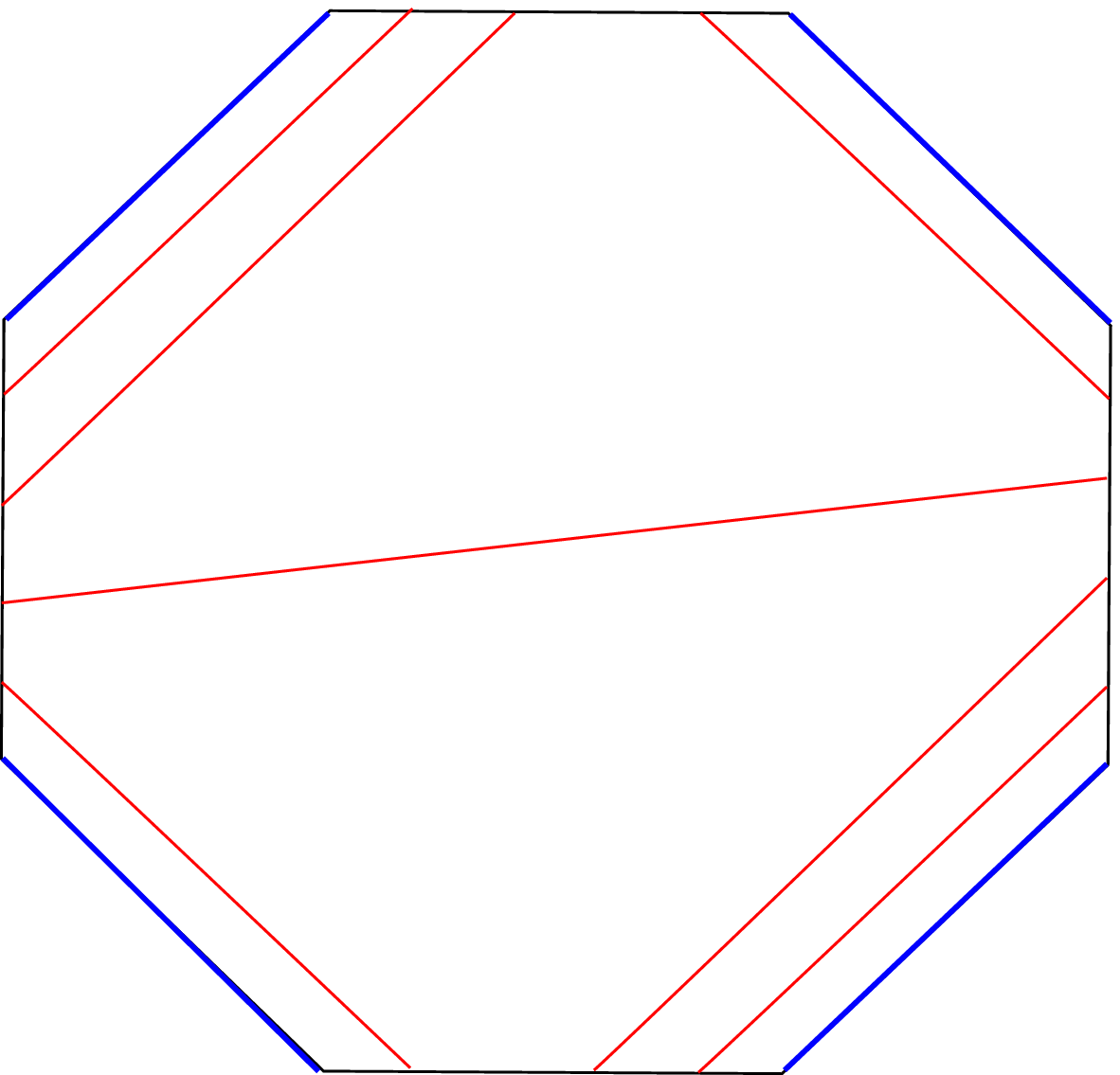}
    \caption{} \label{fig:octagon}
    \end{figure}
    
\section{The construction} \label{sect:construct}

Return now to the notation at the end of Section \ref{sect:standtrip}.  Let $F_P$ be the component of $F - \bdd \Gamma_B^{\Gamma_A}$ whose boundary contains the rectangles in $F_A$ and let $F_Q$ be the component of $F - \bdd \Gamma_B^{\Gamma_A}$ whose boundary contains the rectangles in $F_B$.

\begin{prop} \label{prop:twistedistance}
Suppose $\Gamma_B, \Gamma, \Gamma_A$ is a standard triple of separating meridians in $H$.  Suppose $c$ is a simple closed curve in $F$ that is disjoint from $\Gamma_B$ and, in the solid torus component of $H - \Gamma_B$ in which it lies it is of distant slope.  
\begin{itemize}
\item If $c^{\Gamma_A} \subset F_P$ then $c^{\Gamma_A} \cap F_A$ contains rectangles of slopes $-1$ and $\infty$ and $c^{\Gamma_A} \cap F_B$ contains rectangles with at least two of the four slopes $\infty, 0, 1, \frac 12$ .  
\item If  $c^{\Gamma_A} \subset F_Q$ then $c^{\Gamma_A} \cap F_B$ contains rectangles of slopes $+1$ and $0$ and $c^{\Gamma_A} \cap F_A$ contains rectangles with at least two of the four slopes $-1, \infty, 0, -2$.  
\end{itemize}
\end{prop}  

\begin{proof}  Suppose first that $c^{\Gamma_A} \subset F_P$.  Following Example \ref{example:halftwist}, in order to show that $c^{\Gamma_A} \cap F_A$ contains rectangles of slopes $-1$ and $\infty$ it suffices to show that a transversal of each of the two rectangles of $\bdd \Gamma_B^{\Gamma_A}$ in $F_A$ intersects $c^{\Gamma_A}$ in at least two points.  Untwisting the $\Gamma_A$ twist moves $c^{\Gamma_A}$ back to $c$, moves $\bdd \Gamma_B^{\Gamma_A}$ back to $\Gamma_B$ and moves the transversal of the $-1$-sloped rectangle to a meridional arc in the solid torus $H - \Gamma_B$.  (See the left red arc in Figure \ref{fig:transversals}.)  The simple closed curve $c$ has distant slope, say slope $\frac{r}{s}$, so it intersects a meridian of that torus (and so a meridional arc) in $|s| \geq 2$ points.  Similarly, the transversal to the $\infty$-sloped rectangle twists back to an arc of slope $+1$, so $c$ will intersect this arc in at least $|r - s| \geq 2$ points as well.  (See the other red arc in Figure \ref{fig:transversals}.)

    \begin{figure}[ht!]
  \centering
    \includegraphics[scale=0.5]{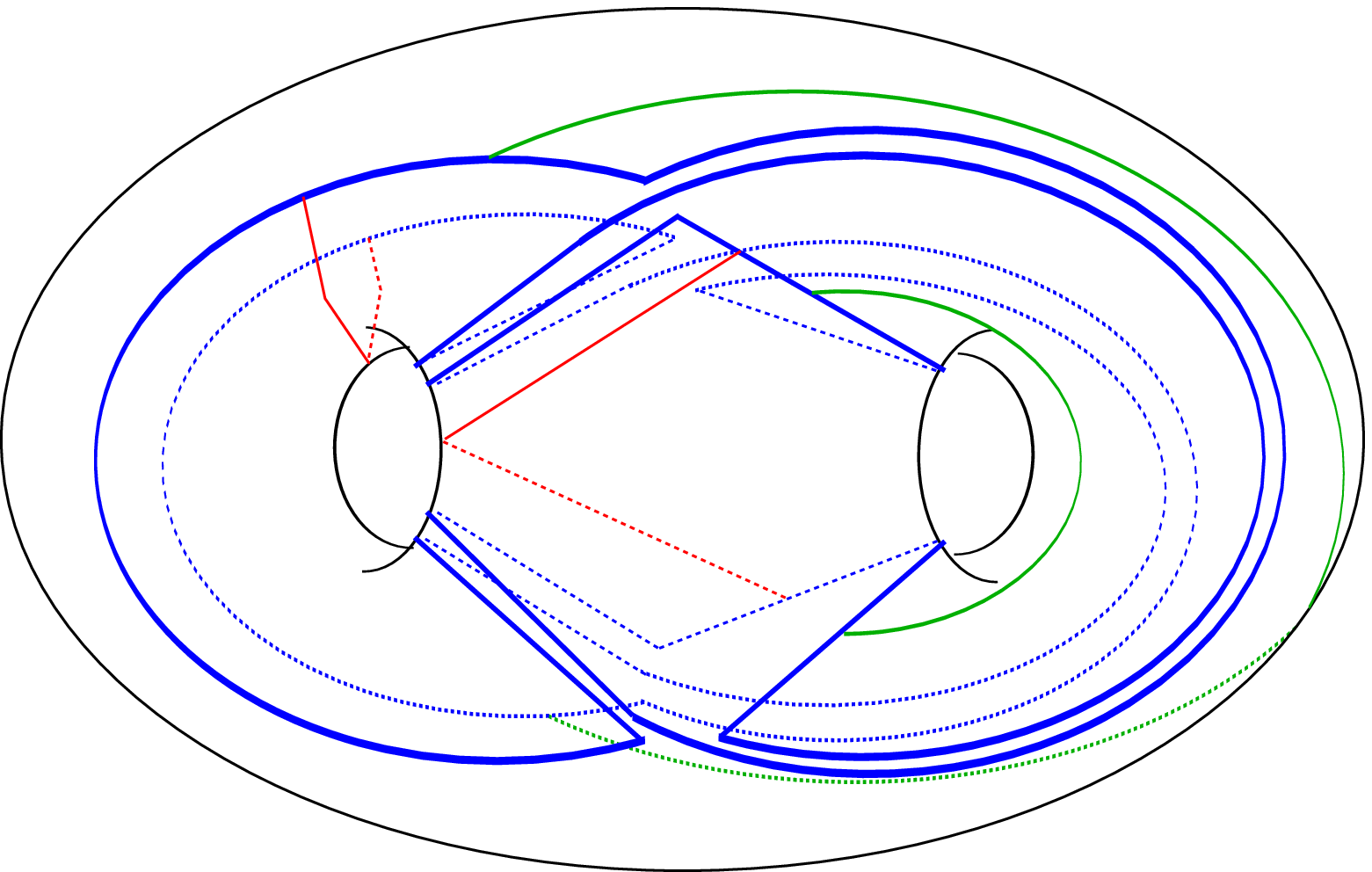}
    \caption{} \label{fig:transversals}
    \end{figure}

To understand $c^{\Gamma_A} \cap F_B$ note that $c^{\Gamma_A}$ lies in the octagon obtained by removing from $F_B$ the two rectangles $F_Q \cap F_B$ or, equivalently, removing from $F_P$ the two rectangles $F_P \cap F_A$.  View the octagon from the latter point of view and invoke Lemma \ref{lemma:octagon}.  We have just shown that $c$ intersects one transversal in $|r - s|$ points and the other in $|s| $ so, according to that lemma, there are  at least $min\{2(|r - s|), 2(|s|)\} \geq 4$ arcs  which have slopes the same as one of the rectangles $F_Q \cap F_B$, that is $0$ or $1$.  Hence there is at least one rectangle with one of those slopes.  Lemma \ref{lemma:octagon} also says there are $||r-s| - |s||$ arcs of $c^{\Gamma_A} \cap F_B$ of a third slope (necessarily then slope $\infty$ or $\frac 12$) in $F_B$.  But $c$ has distant slope in the half of $H - \Gamma_B$ in which it lies, and the definition of this condition is so symmetric that we may as well assume for this computation that $0 < s < r$.  In that case $||r-s| - |s|| = r - 2s \geq 2$ (see Definiton \ref{defin:distant}).  Hence there is a rectangle of $c^{\Gamma_A} \cap F_B$ of slope $\infty$ or $\frac 12$ as well.

If instead $c^{\Gamma_A} \subset F_Q$, the situation is much the same, except the transversals are carried to arcs of slope $0$ and $-1$ instead of $\infty$ and $1$. (See the green arcs in Figure \ref{fig:transversals}.)  But a symmetric argument then works just as well.
\end{proof}

\begin{cor} \label{cor:twistedistance}
Under the assumption of Proposition \ref{prop:twistedistance}, $c^{\Gamma_A} \cap F_A$ and $c^{\Gamma_A} \cap F_B$ each contain rectangles with two distinct close slopes.  
\end{cor}

Begin the construction of a Heegaard split closed $3$-manifold by considering a second genus $2$ handlebody $J$ containing two non-parallel non-separating meridians $X, Y \subset J$, specified longitudes $\lambda_X, \lambda_Y$,  and a corresponding standard triple $\Lambda_X, \Lambda, \Lambda_Y \subset J$ of separating meridians.  There is a natural identification $h_0: J \to H$ given by identifying the named subdisks in their named order, so
\begin{align*} 
h_0(X) &= A, & h_0(Y) &= B, & h_0(\Lambda_X) &= \Gamma_A, &
h_0(\Lambda_Y) &= \Gamma_B, & h_0(\Lambda) &= \Gamma. 
\end{align*}

 Let $h_1: J \to H$ be the composition of $h_0$ with $\frac{2\pi}3$ rotation about the vertical axis  (see Figure \ref{fig:standtrip}). Then
 \begin{align*}
 h_1(\Lambda_Y) &= \Gamma_A, & h_1(\Lambda) &= \Gamma_B, & h_1(\Lambda_X) &= \Gamma, \\
 h_1(\lambda_Y) &= \lambda_A, & h_1(\lambda_{XY}) &= \lambda_B, & h_1(\lambda_X) &= \lambda_{AB}.
 \end{align*}
Follow $h_1$ by a left half-twist along the disk $\Gamma_A$ and call the result $h_2:J \to H$.  Since in each case the homeomorphism is from handlebody to handlebody, the Heegaard splitting $H \cup_{h_i|\bdd J} J, i = 0, 1, 2$ is the same as just doubling the handlebody $H$ along its boundary, i. e. the standard (distance $0$) Heegaard splitting of $(S^1 \times S^2)\#(S^1 \times S^2)$.  

Now alter $J$ by performing Dehn surgery on the cores $c_X, c_Y$ of the two solid torus components $J_X, J_Y$ of $J - \Lambda$.  Choose distant slopes for the surgery and call the result $J^{surg}$.  The disk $\Lambda$ still divides $J^{surg}$ into two solid tori,  which we call $J^{surg}_X$ and $J^{surg}_Y$.  Hence $J^{surg}$ is still a handlebody, but the meridian curve $x \subset \bdd J^{surg}_X$ (resp $y \subset \bdd J^{surg}_Y$) now has distant slope with respect to the original meridian $\bdd X$ (resp $\bdd Y$) and the original longitude $\lambda_X$ (resp $ \lambda_Y$). In particular, since $h_1$ carries $\bdd \Lambda$ to $\Gamma_B$, it follows from Corollary \ref{cor:twistedistance} that in the Heegaard splitting $H \cup_{h_2|\bdd J} J^{surg}$, the two meridians $x$ and $y$ each have rectangles with two distinct close slopes in each of the punctured tori $F_A$ and  $F_B$.  



Let $H^{surg}$ denote the handlebody obtained from $H$ by doing Dehn surgery with remote slopes on both cores $c_A \subset H_A$ and $c_B \subset H_B$.  It follows from Corollary \ref{cor:surgintersect} and Theorem \ref{thm:main} that the Heegaard splitting $H^{surg} \cup_{h_2|\bdd J} J^{surg}$ has distance at least $3$.  

\section{Afterword: creating distinct Heegaard splittings, each of distance $3$} \label{sect:afterword}

Part of the interest in Berge's original construction is that the closed $3$-manifold he constructs also has a second Heegaard splitting, not homeomorphic to the first splitting, but also one of distance $3$.  This discovery illustrates that, in \cite{RS}, the listing  of all possible ways in which a $3$-manifold might have distinct genus $2$ Heegaard splittings is incomplete, see \cite{Be}, \cite{BS}.

So it is interesting to observe that the general Dehn surgery construction described above has the same property: the manifold created from $(S^1 \times S^2)\#(S^1 \times S^2)$ by Dehn surgery on the specified $4$-component link has two (typically different) distance $3$ Heegaard splittings, and the pair of splittings are related in a way that is described in \cite{BS}.  Here is the argument:

In the construction above, the curves $h_1(\lambda_X) = \lambda_{AB}$,$h_1(\lambda_{XY}) = \lambda_B$ and $\bdd \Gamma_A$ are all disjoint in $F$.  Since $h_2$ differs from $h_1$ by a half-twist along $\Gamma_A$ it follows that, even after the half-twist, $h_2(\lambda_X) = \lambda_{AB}$ and $h_2(\lambda_{XY}) = \lambda_B$.  Hence in $H \cup_{h_2|\bdd J} J$ the cores $c_X$ of $J_X$ and $c_B$ of $H_B$ can be isotoped to disjoint curves in $F$ and then past each other,  so that afterwards $c_X \subset H$ and  $c_B \subset J$.  See Figure \ref{fig:schematics} for a highly schematic account of this construction.  After the isotopy, all the core curves are in essentially the same sort of position with respect to each other and with respect to the disk $\Gamma_A$,  so surgery on the link will again alter the splitting to one that is distance $3$. (Technically the surgery slope on $c_X$, now in $H$ and so playing the role of $c_B$, needs to be not just distant but in fact remote.)  An explicit way of seeing the equivalence is to note that reflecting $(S^1 \times S^2)\#(S^1 \times S^2)$ through the separating sphere $\Gamma_A \cup \Lambda_Y$ brings the set of $4$ cores after the exchange to exactly the same set of curves in the original construction; the only difference in the construction of the second splitting is that the reflection changes the left half-twist along $\Gamma_A$ to a right half-twist.

    \begin{figure}[ht!]
 \labellist
\small\hair 2pt
\pinlabel $c_X$ at 0 430
\pinlabel $c_X$ at 290 430
\pinlabel $c_X$ at 380 160
\pinlabel $\Lambda$ at 100 455
\pinlabel $h_0$ at 140 440
\pinlabel $c_Y$ at 250 430
\pinlabel $c_Y$ at 530 430
\pinlabel $c_Y$ at 540 200
\pinlabel $h_1$ at 430 440
\pinlabel $h_1$ at 430 210
\pinlabel $h_1$ at 140 205
\pinlabel $\lambda_X$ at 10 400
\pinlabel $\lambda_A$ at 50 390
\pinlabel $\lambda_A$ at 480 390
\pinlabel $\lambda_B$ at 195 390
\pinlabel $\lambda_B$ at 410 280
\pinlabel $\lambda_Y$ at 240 390
\pinlabel $c_A$ at 100 390
\pinlabel $c_A$ at 440 390
\pinlabel $c_A$ at 440 160
\pinlabel $c_B$ at 420 320
\pinlabel $c_B$ at 150 390
\pinlabel $c_B$ at 440 10
\pinlabel $\Gamma$ at 125 410
\pinlabel $\Lambda_Y$ at 10 300
\pinlabel $\Gamma_B$ at 90 310
\pinlabel $\lambda_{AB}$ at 125 300
\pinlabel $\lambda_{AB}$ at 370 380
\pinlabel $\Gamma_A$ at 160 310
\pinlabel $\Gamma_A$ at 370 330
\pinlabel $\Gamma_A$ at 370 100
\pinlabel $\Lambda_X$ at 230 300
\pinlabel $\lambda_{XY}$ at 125 260
\pinlabel $\lambda_{XY}$ at 410 260
\pinlabel \LARGE{$H$} at 125 370
\pinlabel \LARGE{$J$} at -25 370
  \endlabellist
    \centering
    \includegraphics[scale=0.6]{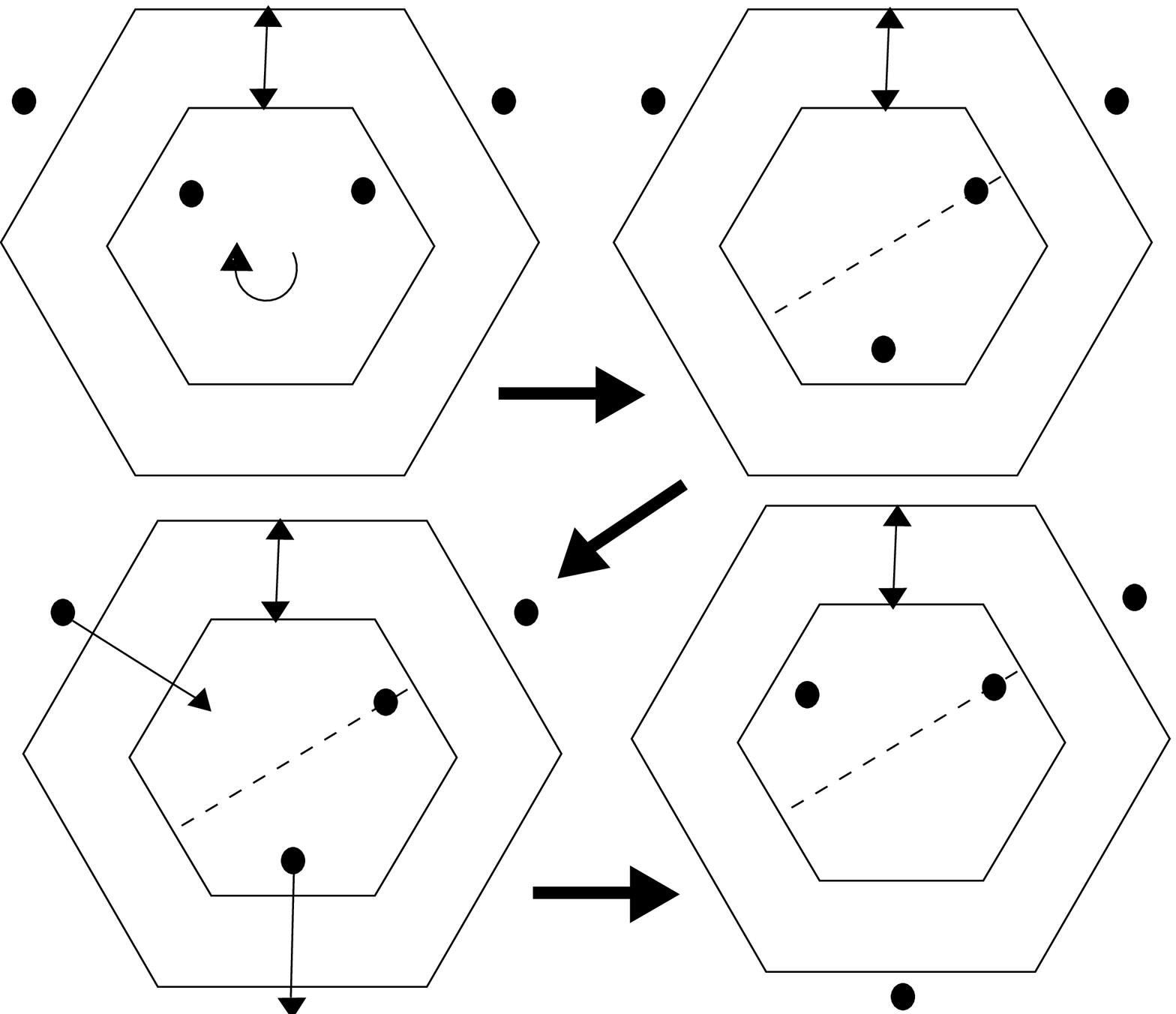}
    \caption{} \label{fig:schematics}
    \end{figure}

The process just described, exchanging two of the cores of the genus two handlebodies before doing Dehn surgery on all four cores, is exactly the construction used in Section 3 of \cite{BS} to describe a manifold $M_H$ with two alternative genus two Heegaard splittings.  (Examples of this sort were shown in \cite{BS} to have the property that a single stabilization makes the two alternative splittings isotopic.)

Berge observes (\cite[Section 5.6]{Be2}) that the discussion of SUMS in Section \ref{sect:SUMS} leads to an easy proof that in most cases the two alternative splittings are not homeomorphic.  First note that the construction is highly symmetric: the roles of $H$ and $J$ can easily be reversed, replacing $h_1: \bdd J \to \bdd H$ by $h_1^{-1}: \bdd H \to \bdd J$ and replacing the half-twist around $\Gamma_A \subset H$ by a half-twist around $\Lambda_Y \subset J$.  So if the surgeries on all cores are chosen to have remote slope then, following Corollary \ref{cor:SUMS}, both the pairs $\{ A, B \}$ and $\{ X, Y \}$ are SUMS for the splitting and so are uniquely defined by the homeomorphism type of the splitting.  In particular, the quadruple of numbers $\{ |\bdd A \cdot \bdd X|, |\bdd A \cdot \bdd Y|, |\bdd B \cdot \bdd X|, |\bdd B \cdot \bdd Y| \}$ is an invariant of the Heegaard splitting up to homeomorphism.  So to show that the two Heegaard splittings of the same manifold are not equivalent, it suffices to check that the quadruple of numbers changes when the two cores are exchanged as above.  This is done for a specific example in \cite{Be2}, where the sum of the four numbers changes from 121 to 149. 

Berge has also pointed out that there are other invariants available to distinguish the splittings, ones that are easier to calculate from a Heegaard diagram or from a Dehn surgery description.  In the case in which the meridians of the two handlebodies are SUMS, as we have here, the separating meridians $\Gamma$ and $\Lambda$ are determined up to homeomorphism.  So, for each arc $\alpha$ of $(\bdd X \cup \bdd Y) \cap F_A$ there is associated a number $|\alpha \cdot \bdd A| \geq 2$.  When all such arcs are accounted for this yields either a pair or a triple of numbers $\{ a_1, a_2, (a_3)\}$ associated to $F_A$.  There is a similar pair or triple of numbers $\{ b_1, b_2, (b_3)\}$ associated to $F_B$ and, from $J$, a further pair of triples or pairs, $\{ \{ x_1, x_2, (x_3)\}, \{ y_1, y_2, (y_3)\} \}$.  The pair of pairs of pairs or triples $\{ \{ \{ a_1, a_2, (a_3)\}, \{ b_1, b_2, (b_3)\} \}, \\ \{ \{ x_1, x_2, (x_3)\} \{ y_1, y_2, (y_3)\} \} \}$ is an invariant for each such splitting, one that can be calculated from a Heegaard diagram or from a Dehn surgery description.



\vspace{10mm}

\baselineskip 14pt \noindent {\sf Department of Mathematics\\ University 
of 
California\\ Santa Barbara, CA 93106\\ mgscharl@math.ucsb.edu}

\end{document}